\documentclass[10pt]{amsart}
\usepackage{graphicx}
\usepackage{latexsym}
\usepackage{fancyhdr}
\usepackage{amsmath, amssymb}
 \usepackage[utf8]{inputenc}
\usepackage[all]{xy}
\usepackage{hyperref}
\usepackage{subfigure}
\usepackage{enumerate}
\usepackage{color}

\theoremstyle{plain}
\newtheorem{thm}{Theorem}%[section]
\newtheorem{cor}[thm]{Corollary}

\newtheorem{lem}[thm]{Lemma} 
\newtheorem{prop}[thm]{Proposition}

\newtheorem{question}[thm]{Question}

\theoremstyle{definition}
\newtheorem{defi}[thm]{Definition}
\newtheorem{rem}[thm]{Remark}
\numberwithin{equation}{section}

\newtheorem{example}[thm]{Example}

\newcommand{\lgw}{\longrightarrow}

\newcommand{\ovl}{\overline}

\newcommand{\ord}{\text{ord}}

\renewcommand{\k}{\Bbbk}

\newcommand{\R}{\mathbb{R}}
\newcommand{\K}{\mathbb{K}}

\newcommand{\N}{\mathbb{N}}

\renewcommand{\L}{\mathbb{L}}

\newcommand{\C}{\mathbb{C}}

\newcommand{\Q}{\mathbb{Q}}

\renewcommand{\lg}{\langle}
\newcommand{\NN}{\aleph_0}
\newcommand{\rg}{\rangle}

\renewcommand{\a}{\alpha}

\renewcommand{\b}{\beta}

\renewcommand{\phi}{\varphi}

\newcommand{\e}{\varepsilon}

\newcommand{\mm}{\mathfrak{ m}}

\begin{document}
\title{Remarks on Artin approximation with constraints}

\author{Dorin Popescu}
\address{Simion Stoilow Institute of Mathematics of the Romanian Academy, Research unit 5,
University of Bucharest, P.O.Box 1-764, Bucharest 014700, Romania}
\email{dorin.popescu@imar.ro}

\author{Guillaume Rond}
\address{Aix-Marseille Universit\'e, CNRS, Centrale Marseille, I2M, UMR 7373, 13453 Marseille, France}
\email{guillaume.rond@univ-amu.fr}

\begin{abstract}
We study various approximation results of solutions of equations $f(x,Y)=0$ where $f(x,Y)\in\K[\![x]\!][Y]^r$ and $x$ and $Y$ are two sets of variables, and where some components of the solutions $y(x)\in\K[\![x]\!]^m$ do not depend on all the variables $x_j$. These problems were highlighted by M. Artin.
\end{abstract}


\maketitle

\section{Introduction}

Let $(R,\mm)$ be a Henselian  excellent Noetherian local ring, $f=(f_1,\ldots,f_r)$ a system of polynomials in $Y=(Y_1,\ldots,Y_m)$ over $R$ and $\hat{y}$ a zero of $f$ in the completion $\hat{R}$ of $R$.
\begin{thm} (Popescu \cite{P}, \cite{P2},  Swan \cite{S})\label{po} For every $c\in \N$ there exists a zero $y$ of $f$ in $R$ such that $y\equiv \hat{y}$ modulo $\mm^c$.
\end{thm}

M. Artin proved in \cite[Theorem 1.10]{Ar0} the most important case of this theorem, that is when $R$ is the algebraic power series ring in $x=(x_1,\ldots,x_n)$ over a field $\K$. Usually we rewrite  Theorem \ref{po} saying that excellent Henselian local rings have the Artin approximation property.

 Now suppose that $\hat{R}$ is the formal power series ring in $x=(x_1,\ldots,x_n)$ over a field $\K$ and some components of $\hat{y}$ have some constraints, that is they  depend  only on some of  the variables $x_j$. M. Artin asked if it is possible to find $y\in R^m$ such that the correspondent components depend on the same variables $x_j$  (see  \cite[Question 4]{Ar}). More precisely, we have the following question. For a set $J\subset [n]$ we denote by $\K[\![x_J]\!]$ the ring of formal power series in the $x_j$ for $j\in J$.
 
 \begin{question}\label{q}
 (Artin Approximation with constraints \cite[Problem 1, page 68]{R}) Let $R$ be an excellent local subring of $\K[\![x]\!]$, $x=(x_1,\ldots,x_n)$ such that the completion of $R$ is $\K[\![x]\!]$ and $f\in R[Y]^r$, $Y=(Y_1,\ldots,Y_m)$. Assume that there exists a formal solution $\hat{y}\in \K[\![x]\!]^m$ of $f=0$ such that $\hat{y}_i\in  \K[\![x_{J_i}]\!]$ for some subset $J_i\subset [n]$, $i\in [m]$. Is it possible to approximate $\hat{y}$  by a solution $y\in R^m$ of $f=0$ such that
$y_i\in R\cap \K[\![x_{J_i}]\!]$, $i\in [m]$?
 \end{question}
 
If $R$ is the algebraic power series ring in $x=(x_1,x_2,x_3)$ over $\C$ then Becker \cite{Be} gave a counterexample. If the set $(J_i)$ is totally ordered by inclusion, that is the so called  Nested   Artin Approximation then this question has a positive 
answer in \cite{P}, \cite[Corollary 3.7]{P2} (see also \cite[Theorem 3.1]{CPR} for an easy proof in the linear case). However, when $R$ is the convergent power series ring in  $x=(x_1,x_2,x_3)$ over $\C$ then Gabrielov \cite{Ga} gave a counterexample (see also \cite{Iz} for a general account on this problem).
 
 A field extension $\K\subset \K'$ is algebraically pure (see \cite{P1}, \cite{BNP})
if every finite system of polynomial equations has a solution in $\K$ if it has one in $\K'$. Any field extension of an algebraically closed field is algebraically pure \cite{P1}. In connection with Question  \ref{q} the following theorem was proved.

\begin{thm}(Kosar-Popescu \cite[Theorem 9]{KP})\label{t}
Let $\K \to \K'$ be an algebraically pure morphism of fields and $x=(x_1, \ldots , x_n)$. Let $J_i$, $i\in [m]$ be  subsets of $[n]$,  and $A_i=\K\langle x_{J_i}\rangle$, resp. $A'_i=\K'\langle x_{J_i}\rangle$, $i\in [m]$ be the algebraic power series in $x_{j_I}$ over $\K$ resp. $\K'$. Set
 $\mathcal{N}= A_1 \times \cdots \times A_m$ and
  $\mathcal{N'}=A'_1 \times \cdots \times A'_m$. Let $f$ be a system of polynomials from $\K\langle x\rangle [Y]$, $Y=(Y_1,\ldots,Y_m)$, and ${\hat y}\in \mathcal{N'}$,  such that $f({\hat y})=0$.
Then there exist $y\in \mathcal{N}$  such that $f(y)=0$ and   $\ord( y_i)=\ord (\hat y_i)$ for $i\in [m]$.
\end{thm}

The goal of our paper is to replace somehow in Theorem \ref{t}
the algebraic power series by formal power series (see Theorem \ref{q2}) and to state a certain Artin strong approximation with constraints property of the formal power series ring in $x$ over a field $\K$ which is so-called $\NN$-complete (see Corollary \ref{cor}).
This condition on $\K$ is necessary (see Remarks \ref{r1}, \ref{r2}). Finally we apply these results to extend approximation results due to J. Denef and L. Lipshitz for differential equations with coefficients in the ring of univariate polynomials to the case of several indeterminates (see Corollaries \ref{cor1} and \ref{cor2}).

 Finite fields, uncountable algebraically closed fields and ultraproducts of fields over $\N$ are $\NN$-complete (see Theorem \ref{t0}). If $(\K_n)_n$ is a sequence of fields and $\mathcal F$ is an ultrafilter of $\N$ we denote by $(\K_n)^*$ the ultraproduct (over the natural numbers) defined as $\left(\prod_{n\in\N}\K_n\right)/\mathcal F$, that is the factor of $\left(\prod_{n\in\N}\K_n\right)$ by the ideal $\{(x_n)_{n\in \N}\in \left(\prod_{n\in\N}\K_n\right) :\{n\in \N:x_n=0\}\in \mathcal{F}\}$. When $\K$ is a single field, $\K^*$ denotes the ultrapower   $\left(\prod_{n\in\N}\K\right)/\mathcal F$.
%%%%%%%%%%%%%%%%%%%%%%%%%%%%%%%%%%%%%%%%

\section{Solutions of countable systems of polynomial equations}

\begin{defi}
Let $\K$ be a field. We say that $\K$ is \emph{$\NN$-complete} if every countable system $\mathcal S$ of polynomial equations (in a countable number of indeterminates) has a solution in $\K$ if and only if every finite sub-system of $\mathcal S$ has a solution in $\K$.
\end{defi}

\begin{thm}\label{t0}
The following fields are $\NN$-complete:
\begin{enumerate}
\item[a)] Every finite field. 
\item[b)] Every uncountable algebraically closed field.
\item[c)] Every ultraproduct of fields over the natural numbers.
\end{enumerate}
\end{thm}

\begin{rem}
Every ultraproduct is either finite or uncountable. So every algebraically closed field which is an ultraproduct is necessarily uncountable.
\end{rem}

\begin{proof}
Let $\mathcal S$ be a system of countably many polynomial equations with coefficients in a field $\K$. We list the polynomial equations of $\mathcal S$ as $P_1,\ldots, P_n,\ldots$ which depends on the variables $x_1,\ldots, x_l,\ldots$.\\
For any $N\in\N$ let $D_N$ be an integer  such that the polynomials $P_i$, for $i\leq N$, depend only on the $x_j$ for $j\leq  D_N$.\\
Let us define the canonical projection maps:
$$\pi_{l,k} : \K^l=\K^{k}\times \K^{l-k}\lgw \K^k \ \ \forall l\geq k\geq 1$$
that sends the vector $(x_1,\ldots, x_l)$ onto $(x_1,\ldots, x_k)$. We also define the projection maps
$$\pi_k :\K^\N\lgw \K^k\ \ \ \forall k\geq 1$$
that send the sequence $(x_1,\ldots, x_n,\ldots)$ onto $(x_1,\ldots, x_k)$.\\
Let 
$$V_\infty:=\{x=(x_n)_n\in\K^\N\mid P_i(x)=0\ \forall i\in\N\}$$
and 
$$V_{N}:=\{x=(x_n)_n\in\K^\N\mid P_1(x)=\ldots=P_N(x)=0\} \ \ \forall N\in\N.$$
Then we have that $V_\infty=\cap_{N\in\N}V_N$.
By assumption, for every integer $N\geq 1$ we have that
$$V_N=\pi_{D_N}(V_N)\times\K^{\N\setminus \{1,\ldots,D_N\}}.$$
For every positive integers $N$ and $k$ we define
$$C_N^k=\pi_{k}(V_N).$$
Now set 
$$C^k:=\cap_{N\in\N}C_N^k.$$ We claim that, if for every $k$, $C_k\neq \emptyset$, then $\mathcal S$ has a solution; indeed, by construction $(x_1,\ldots, x_k)\in C^k$ if and only if for every $N$ and $k$ there exists $(x_{k+1},\ldots, )\in\K^\N$ such that $(x_1,\ldots, x_k,x_{k+1},\ldots)\in V_N$. In particular $\pi_{k+1,k}(C^{k+1})=C^k$ for every $k$.\\
Now let $x_1\in C^1$. Then there exists $x_2\in\K$ such that $(x_1,x_2)\in C^2$. By induction we can find a sequence of elements $x_n\in\K$ such that for every $k$ 
$$(x_1,\ldots, x_k)\in C^k.$$
Thus the sequence $x=(x_n)_n\in V_N$ for every $N$ so it belongs to $V_\infty$. Hence $\mathcal S$ has a solution.\\
\\
a) Let us assume that $\K$ is a finite field.\\
Then the $C_N^k$ are finite subsets of $\K^k$. Since $V_{N+1}\subset V_N$ for every $N$, the sequence $(C_N^k)_N$ is decreasing so it stabilizes. Therefore $C_k\neq \emptyset$  and $\mathcal S$ has a solution.\\
\\b) Now let us assume that $\K$ is an uncountable algebraically closed field. 
We have that 
$$C_N^k=\pi_k(V_N)=\pi_{D_N,k}\left(\{x=(x_1,\ldots, x_{D_N})\in\K^{D_N}\mid P_1(x)=\ldots=P_N(x)=0\}\right).$$
Thus the $C_N^k$ are constructible subsets of $\K^k$ since $\K$ is algebraically closed (by Chevalley's Theorem). Let us recall that a constructible set is a finite union of sets of the form $X\backslash Y$ where $X$ and $Y$ are   Zariski closed  subsets of $\K^{k}$.\\
Thus the sequence $(C_N^k)_N$ is a decreasing sequence of constructible subsets of $\K^k$. Let $F_N^k$ denote the Zariski closure of $C_N^k$. Then the sequence $(F_N^k)_N$ is a decreasing sequence of Zariski closed subsets of $\K^k$. By Noetherianity this sequence stabilizes, i.e. $F_{N}^k=F_{N_0}^k$ for every $N\geq N_0$ and some positive integer $N_0$. By assumption $C_{N_0}^k\neq \emptyset$  so $F_{N_0}^k\neq \emptyset$. Let $F$ be an irreducible component of $F_{N_0}^k$. \\
Since $C_N^k$ is constructible,  $C_N^k=\cup_i\left(X^N_i\backslash Y^N_i\right)$ for a finite number of Zariski closed sets $X^N_i$ and $Y^N_i$ with $X^N_i\backslash Y^N_i\neq \emptyset$ and $X^N_i$ is assumed irreducible. Since $X_i^N$ is irreducible the Zariski closure of $X^N_i\backslash Y^N_i$ is $X_i^N$. Therefore for $N\geq N_0$ we have that
$$F_{N_0}^k=F_N^k=\cup_i X^N_i.$$
 But $F$ being irreducible, for every $N\geq N_0$ one of the $X_i^N$ has to be equal to $F$. Thus for every $N\geq N_0$ there exists a closed proper subset $Y_N\subset F$ such that
$$F\backslash Y_N\subset C_N^k\ \ \forall N\geq N_0.$$
Since $\K$ is uncountable 
$$\bigcup_{N\geq N_0}Y_N\subsetneq F.$$
This is a well known fact (see for instance Exercice 5.10, \cite{L} p. 76). This implies that $C^k\neq \emptyset$
 and $\mathcal S$ has a solution.\\
\\
Finally c) is given as in Lemma 2.17 \cite{P1}.
\hfill\ \end{proof}

\begin{rem}
It is quite straightforward to prove that a field $\K$ that is $\aleph_1$-saturated  is $\NN$-complete (for the definition of a saturated model see \cite[Section 2.3]{CK}). One can prove that the three fields of Theorem \ref{t0} are $\aleph_1$-saturated providing an alternative proof of the fact that these fields are $\NN$-complete.
\end{rem}

%%%%%%%%%
\begin{example} Let $\K=\ovl\Q$ be the algebraic closure of $\Q$.
Since $\ovl\Q$ is countable we may list its elements as $\a_1$, $\a_1$, \ldots, $\a_l$, \ldots.
Let $\mathcal S$ be the system of equations:
$$P_1=0,P_l=\left(x_1-\a_{l}\right)x_{l}-1= 0 \ \ \forall l\geq 2.$$
For every integer $N\geq 1$ the vector 
$$\left(\a_N, \frac{1}{\a_N-\a_{2}}, \ldots, \frac{1}{\a_N-\a_{N-1}}\right)\in\K^{N-1}$$
is a solution of
$$P_1=\cdots= P_{N-1}=0.$$
But $\mathcal S$ has no solution. Indeed if $x=(x_1,\ldots, x_n,\ldots)\in\K^\N$ was a solution of $\mathcal S$ then we would have that
\begin{equation}\label{eq}(x_1-\a_{l})x_{l}=1\ \ \ \forall l\geq 2.\end{equation}
But $x_1\in\ovl\Q$ so $x_1=\a_{l_0}$ for some $l_0\geq 0$. Thus \eqref{eq} for $l=l_0$ would give
$$0=(x_1-\a_{l_0})x_{l_0}=1$$
which is impossible. So $\ovl\Q$ is not an $\NN$-complete field.
\end{example}
%%%%%%%%%%%%%%%%%

\begin{example}
Let $\K=\R$ be the field of real numbers. Let $\mathcal S$ be the system of equations:
$$P_1=0, P_l=x_l^2-(x_1-l)=0\ \ \forall l\geq 2.$$
Then $P_1=\cdots =P_l=0$ has a solution $x=(x_1,\ldots,x_n)$ if and only if 
$x_1-l\geq 0.$\\
In particular $\mathcal S$ has no solution. So $\R$ is not an $\NN$-complete field.
\end{example}

%%%%%%%%%%%%%%%%%%%%%%%%%%%%%%%%%%%%%%%%%%%%%
%%%%%%%%%%%%%%%%%%%%%%%%%%%%%%%%%%%%%%%%%%%%%

\section{Approximation with constraints}

We recall some elementary facts on algebraically pure field extensions, referring to \cite{P1} and \cite[(2.3)]{BNP} for details.

\begin{rem}
\begin{enumerate}

\item If $\K\lgw \L$ is a field extension of real closed fields then it is algebraically pure.
\item If $\K$ is an infinite field and $x=(x_1,\ldots,x_n)$ then $\K\lgw \K(x)$ is algebraically pure. \cite{P1}
\item If $\K$ is a field and $x=(x_1,\ldots, x_n)$, we denote by $\K\lg\!\lg x\rg\!\rg$ the field of algebraic power series, and by $\K\{\!\{x\}\!\}$ the field of convergent power series (when $\K$ is a complete valued field). Then $\K\lg\!\lg x\rg\!\rg\lgw \K\{\!\{x\}\!\}$ and $\K\{\!\{x\}\!\}\lgw \K(\!(x)\!)$ are algebraically pure by Artin approximation theorem. \cite{Ar0}
\item If $\K_1\lgw \K_2$ and $\K_2\lgw \K_3$ are algebraically pure then $\K_1\lgw \K_3$ is algebraically pure. \cite{P1}
\end{enumerate}
\end{rem}

\begin{lem}\label{lem_ultra}\cite{BNP}
Let $\K$ be a field and let $\K^*$ be an ultrapower of $\K$. Then the morphism $\K\lgw \K^*$ sending every element $a\in\K$ onto the constant sequence $(a,\ldots, a,\ldots)$ is algebraically pure. 
\end{lem}

\begin{proof}
Let $\mathcal S=(P_i)_{i\in I}$ be a finite system of polynomial equations with coefficients in $\K$ in the indeterminates $Y_1,\ldots, Y_m$. Let us assume that there exists $y^*\in(\K^*)^m$ such that
$$P_i(y^*)=0\ \ \ \forall i\in I.$$
Let  $(y_n)_{n\in\N}\in(\K^m)^\N$ be a sequence whose image in $(\K^*)^m$ is $y^*$. There for every $i\in I$ there exists $\mathcal U_i\in \mathcal F$ (here $\mathcal F$ denotes the ultrafilter such that $\K^*=\K^\N/\mathcal F$) such that
$$\forall n\in U_i,\ \ P_i(y_n)=0.$$
Since $I$ is finite the intersection $\mathcal U:=\cap_{i\in I}\mathcal U_i\in\mathcal F$. Thus for every $n\in \mathcal U$ we have that
$$P_i(y_n)=0 \ \ \forall i\in I.$$
Hence $\mathcal S$ has a solution in $\K^m$. Therefore $\K\lgw \K^*$ is algebraically pure.
\end{proof}

\begin{prop}\label{prop2}
Let $\K$ be a $\NN$-complete field. Let $x=(x_1,\ldots, x_n)$,  $Y=(Y_1,\ldots, Y_m)$, $f=(f_1,\ldots,f_r)\in \K[\![x]\!][Y]^r$ and $J_i\subset [n]$, $i\in [m]$.\\
If for every $c\in\N$ there exists $y^{(c)}\in\K[\![x]\!]^m$, with $y^{(c)}_i\in \K[\![x_{J_i}]\!]$ for every $i$, such that
$$f(y^{(c)})\equiv 0\ \mbox{modulo}\ (x)^c$$
then there exists $y\in\K[\![x]\!]^m$, with $y_i\in \K[\![x_{J_i}]\!]$ for every $i$,  such that
$$f(y)=0.$$

\end{prop}

\begin{proof} Let us set $$B_i:=\N^{\e_{1,i}}\times\cdots \times \N^{\e_{m,i}}$$
where $\e_{k,i}=1$ if $k\in J_i$, $\e_{k,i}=0$ if $k\notin J_i$, and 
$$Y_i=\sum_{\a\in B_i}Y_{i,\a}x^\a \ \ \ \forall i=1,\ldots, m.$$
We denote by $P_{k,\b}$ the coefficient of $x^\b$ in $f_k(\sum_{\a\in B_1}Y_{1,\a}x^\a,\ldots,\sum_{\a\in B_m}Y_{m,\a}x^\a)$. 
 Let us denote by $\mathcal S$ the system of polynomial equations
\begin{equation}\label{2}P_{k,\b}=0,\ k\in[p],\ \b\in \N^n.\end{equation}
depending on the variables $Y_{i,\a}$ for $i\in[m]$ and $\a\in  B_i$.\\
Since $\K$ is a $\NN$-complete field and every finite sub-system of $\mathcal S$ has a solution,  $\mathcal S$ has a solution $(y_{i,\a})_{i\in[m],\a\in B_i}$ with coefficients in  $\K$. Thus if $y=(y_1,\ldots, y_m)$ with
 $$y_i=\sum_{\a\in B_i} y_{i,\a}x^\a$$
 then we have that $f(y)=0$.
\hfill\ \end{proof}

\begin{example}
In \cite{BDLD} two examples are given that show that this statement is no longer true without the condition of $\K$ being $\NN$-complete: the first one is a system of polynomial equations over the algebraic closure of $\mathbb F_p$ (see Example (i) p. 200 \cite{BDLD}) and the second one is an example of polynomial equations over $\Q$ (see Example (ii) p. 200 \cite{BDLD}).\\
\end{example}

%%%%%%%%%%%%%%%%%%%%%%%%%%%%%%%%%%%%%%%%%%%%%%%

\begin{thm}\label{q2} Let $\K\subset \K'$ be an algebraically pure field extension where $\K$ is $\NN$-complete. We set  $x=(x_1,\ldots,x_n)$ and $f\in \K[\![x]\!][Y]^r$, $Y=(Y_1,\ldots,Y_m)$.\\
Assume that there exists a  solution $\hat{y}\in \K'[\![x]\!]^m$ of $f=0$ such that
 $$\hat{y}_i\in  \K'[\![x_{J_i}]\!]$$
   for some subsets
 $J_i\subset [n]$, $i\in [m]$. Then there is a solution $y\in \K[\![x]\!]^m$ of $f=0$ such that
$y_i\in  \K[\![x_{J_i}]\!]$ and  $\ord (y_i)=\ord (\hat{y}_i)$, $i\in [m]$.

\end{thm}

\begin{proof}
Let us write $\hat y_i=\sum_{\a\in B_i}\hat y_{i,\a}x^\a$ where 
$B_i\subset \N^n$ denotes the support of $\hat y_i$.

We have that
$$f(\hat y)=0\Longleftrightarrow f_k(\hat y  )=0 \ \ \forall k=1,\ldots, r$$
$$\Longleftrightarrow  \forall k, \ \forall \b\in\N^n \text{ the coefficient of } x^\b \text{ in } f_k(\hat y) \text{ is } 0$$
Let us denote by $P_{k,\b}$ the coefficient of  $x^\b$ in  $f_k$ after replacing each $Y_i$ by the term $\sum_{\a\in B_i}Y_{i,\a}x^\a$, and let 
 $\mathcal S$ be the system of equations
 $$P_{k,\b}=0\ \ \forall k\in\N, \ \forall \b\in\N^n $$
 in the indeterminates $Y_{i,\a}$ for $i=1,\ldots, m$ and $\a\in B_i$. Since $\mathcal S$ has a solution in $\K'$ every finite sub-system of $\mathcal S$ has a solution in $\K'$ and, since $\K\lgw \K'$ is algebraically pure,  every finite sub-system of $\mathcal S$ has a solution in $\K$. Then, since $\K$ is a $\NN$-complete field the system $\mathcal S$ has a solution $(y_{i,\a})_{i\in[m],\a\in B_i}$ with coefficients in  $\K$. This means that if $y=(y_1,\ldots, y_m)$ with
 $$y_i=\sum_{\a\in B_i} y_{i,\a}x^\a$$
 then $f(y)=0$. Since $B_i$ is the support of $\hat y_i$,the support of $y_i$ is included in the support of $\hat y_i$ for every $i$. In particular we have that $\ord(\hat y_i)\leq \ord(y_i)$ for every $i$.\\
 Now let us assume moreover  that $\ord(\hat y_i)=c_i$ and that, for every $i=1,\ldots, m$, $\hat y_{i,\a_i}\neq 0$ with $|\a_i|=c_i$ (here for $\b=(\b_1,\ldots, \b_n)$ we set $|\b|:=\b_1+\cdots+\b_n$). Then there exists, for $i=1,\ldots, m$, an element $\hat z_i\in\K'$ such that
 $$\hat y_{i,\a_i}\hat z_i=1, \ \forall i=1,\ldots, m.$$
 By adding the equations 
 \begin{equation}\label{eq}Y_{i,\a_i}Z_i=1, \ \forall i=1,\ldots, m.\end{equation}
 to the system $\mathcal S$ we can suppose that there exists $z_i\in\K$ for every $i$ such that Equations \eqref{eq} are satisfied. Thus
 $$\ord(y_i)=c_i=\ord(\hat y_i)\ \ \forall i=1,\ldots,m$$
 and the theorem is proven.
\hfill\ \end{proof}

\begin{rem}\label{r1}
By Lemmas 5.1 and 5.2 \cite{R} every system $\mathcal T$ of partial polynomial differential equations with  coefficients in $\K[\![x]\!]$ (with $x=(x_1,\ldots, x_n)$) and indeterminates $Y_1$, \ldots, $Y_m$, provides a system $\mathcal S$ of polynomial equations with  coefficients in $\K[\![x]\!][t]$ (with $t=(t_1,\ldots, t_l)$) and indeterminates $Y_1$, \ldots, $Y_m$, $Z_{1}$,\ldots, $Z_{k}$ such that $y\in\K[\![x]\!]^m$ is a solution of $\mathcal T$ if and only if there exists $z\in\K[\![x,t]\!]^k$ such that $(y,z)$ is a solution of $\mathcal S$ and $z$ satisfies some constraints conditions as in Proposition \ref{prop2}.\\
By Corollary 4.7 \cite{DL} there exists a system of partial differential equations $\mathcal T$ defined over $\ovl\Q$ having a solution whose components are  in $\C[\![x]\!]$ but no solution  whose components are in $\ovl\Q[\![x]\!]^m$. So it shows that there exists a system of polynomial equations $\mathcal S$ with coefficients in $\ovl\Q[x]$ which has no solution $y\in\ovl\Q[\![x]\!]^m$ such that $y_i\in\ovl\Q[\![x_{J_i}]\!]$ for every $i$ for some $J_i\subset [n]$, but has a solution $y'\in\C[\![x]\!]^m$ such that $y'_i\in\C[\![x_{J_i}]\!]$ for every $i$.\\
This shows that Theorem \ref{q2} is no longer true in general if  $\K$ is not $\NN$-complete.\\
Moreover since  this system $\mathcal S$ has a  solution with coefficients in $\C$ satisfying the constraints conditions and since $\ovl\Q\lgw \C$ is algebraically pure, for every $c\in\N$ there exists $y^{(c)}\in\ovl\Q[\![x]\!]^m$ (satisfying the constraints conditions) such that $f(y^{(c)})\in (x)^c$. But there is no $y\in\ovl\Q[\![x]\!]^m$ (satisfying the constraints conditions) such that $f(y)=0$. This also provides an example showing that Proposition \ref{prop2} is not true if $\k=\ovl\Q$.
 \end{rem}

%%%%%%%%%%%%%%

\begin{cor}\label{cor}
Let $\K$ be a $\NN$-complete field. Let us set   $x=(x_1,\ldots,x_n)$,  $f=(f_1,\ldots,f_r)\in \K[\![x]\!][Y]^r$, $Y=(Y_1,\ldots,Y_m)$   and  $J_i\subset [n]$, $i\in [m]$.
Then there exists a map $\nu:\N^m\to \N $ such that if $y'=(y'_1,\ldots,y'_m)$,  $y'_i\in \K[\![x_{J_i}]\!]$, $i\in [m]$ satisfies  $f(y')\equiv 0$ modulo $(x)^{\nu(c)}$ for some $c=(c_1,\ldots,c_m)\in \N^m$ and $\ord (y'_i)=c_i$, $i\in [m]$ then there exists $y_i\in \K[\![x_{J_i}]\!]$ for all $ i\in [m]$ such that $y=(y_1,\ldots,y_m)$ is a zero of $f$  and  $\ord( y_i)=c_i$ for all $i\in [m]$.
\end{cor}

\begin{proof}  Let $c$ be as above. For proof by contradiction suppose that for each $q\in \N$ there exists ${\hat y}_q\in \K[\![x]\!]^m$ with $f({\hat y})\equiv 0$ modulo $x^{q}$, ${\hat y}_{q,i}\in \K[\![x_{J_i}]\!]$, $\ord (\hat y_{q,i})=c_i$, but there exists no solution $y'$ in $\K[\![x]\!]$ with $ y'_{i}\in \K[\![x_{J_i}]\!]$, $\ord (y'_{i})=c_i$. Then let us define $y^*_i=[(y_{qi})_q]\in \K[\![x_{J_i}]\!]^*$. So we have that 
$f(y^*)\in \cap_q x^q\K[\![x]\!]^*$. Set $\bar {y}= y^*$ modulo $\cap_q x^q\K[\![x]\!]^*$ which corresponds to an element in $\K^*[\![x]\!]$ with $f(\bar{y})=0$ (see Lemma 3.4 \cite{BDLD}), $\ord (\bar {y}_i)=c_i$ and $\bar{y}_i\in \K^*[\![x_{J_i}]\!]$.  By Lemma \ref{lem_ultra} and Theorem \ref{q2} there exists $y\in \K[\![x]\!]^m$ with $f(y)=0$,
$\ord (y_i)=c_i$ and $y_i\in \K[\![x_{J_i}]\!]$.  We obtain a contradiction, so the theorem is true.
\hfill\ \end{proof}

\begin{rem} \label{r2} In  Example (iii) p. 201 \cite{BDLD} an example of a system of polynomial equations over $\C$ with constraints is given for which the following is shown: there is no $\nu\in \N$ such that if there exists $\hat y\in \C[\![x]\!]^m $ with  $f(x,\hat y)\in (x)^\nu$ with the given constraints then there exists a solution $y\in\C[\![x]\!]$ of $f=0$ with same constraints and such that $y\equiv \hat y$ modulo $(x)$.
\end{rem}
%%%%%%%%%%%%%%%%%%%%%
\section{Approximation for differential equations}

\begin{cor}\label{cor1}
Let $\K$ be a $\NN$-complete field. Let $F$ be a system of polynomial  equations in $z_1,\ldots,z_q$ and some of  their differentials $\partial^{|j_1|} z_{i_1}/\partial x^{j_1},\ldots, \partial^{|j_s|} z_{i_s}/\partial x^{j_s}$, $i_1,\ldots,i_s\in [q]$, and $j_1,\ldots,j_s\in {\bf N}^n$, with coefficients in $\K[\![x]\!]$. If $F=0$ has approximate solutions up to any order then $F=0$ has a solution with coefficients in $\K[\![x]\!]$.

\end{cor}

\begin{proof}
Exactly as in Remark \ref{r1}, Lemmas 5.1 and 5.2 \cite{R} show that for such a system $F=0$ there is a system of polynomial equations $G=0$ with  coefficients in $\K[\![x]\!][t]$ (with $t=(t_1,\ldots, t_l)$) and indeterminates $Y_1$, \ldots, $Y_m$, $Z_{1}$,\ldots, $Z_{k}$ such that $y\in\K[\![x]\!]^m$ is a solution of $F=0$ if and only if there is $z\in\K[\![x,t]\!]^k$ such that $(y,z)$ is a solution of $G=0$ with constraints.\\
Moreover $y\in\K[\![x]\!]^m$ is an approximate  solution of $F=0$ up to order $c$ if and only if there is $z\in\K[\![x,t]\!]^k$ such that $(y,z)$ is an approximate solution of $G=0$ up to degree $c$ with constraints. This shows that Proposition \ref{prop2} implies Corollary \ref{cor1}.
\hfill\ \end{proof}

\begin{rem}
This theorem has been proven in \cite{DL} in the case of a single indeterminate $x$ under some different hypothesis on $\K$, namely $\K$ has to be a characteristic zero field which is either algebraically closed, a real closed field or a Henselian valued field. Still in \cite{DL} they remark that this theorem is quite easy to prove when $\K=\C$. \\
Again in \cite{DL} is given an example of a system of partial differential equations with coefficients in $\R[\![x_1,\ldots,x_n]\!]$ for $n\geq 2$ having approximate solution up to any degree, but no exact solution (see Corollary 4.10 \cite{DL}). And Corollary 4.7 \cite{DL} provides an analogous example in the case where $\K=\ovl\Q$. These examples show that the univariate case and the case of several variables $x$ are different.
\end{rem}

\begin{cor}\label{cor2}
Let $\K$ be a $\NN$-complete field. Let $F$ be a system of differential equations in $z_1,\ldots,z_q$ and some of  their differentials $\partial^{|j_1|} z_{i_1}/\partial x^{j_1},\ldots, \partial^{|j_s|} z_{i_s}/\partial x^{j_s}$, $i_1,\ldots,i_s\in [q]$, and $j_1,\ldots,j_s\in {\bf N}^n$ with coefficients in $\K[\![x]\!]$. Then  there exists a map $\tau:{\bf N}^{q+s}\to {\bf N} $ such that if $z'=(z'_1,\ldots,z'_q)$,   satisfies  $$F(z',\partial^{|j_1|} z'_{i_1}/\partial x^{j_1},\ldots,
\partial^{|j_s|} z_{i_s}/\partial x^{j_s})\equiv 0 \ \mbox{modulo}\  (x)^{\tau(c)}$$
 for some $c=(c_1,\ldots,c_q,c_{i_1,j_1},\ldots,c_{i_s,j_s})\in {\bf N}^{q+s}$ and $\ord (z'_i)=c_i$, $i\in [q]$, 
  $$\ord \left(\frac{\partial^{|j_k|} z'_{i_k}}{\partial x^{j_k}}\right)=c_{i_k,j_k},$$
   $k\in [s]$ then there exists  $z=(z_1,\ldots,z_q)\in \K[\![x]\!]^q$    a solution of $F$  together with its corresponding differentials such that  $\ord( z_i)=c_i$ for all $i\in [q]$ and  
   $$\ord \left(\frac{\partial^{|j_k|} z_{i_k}}{\partial x^{j_k}}\right)=c_{i_k,j_k}, \ \ k\in [s].$$
\end{cor}

\begin{proof} Let $f\in \K[\![x]\!][Y]^r$, $Y=(Y_1,\ldots,Y_m)$, $m>q+s$ be the transformation of $F$ in an algebraic system of equations with constraints as done in the proof of Corollary \ref{cor1}. Assume that $z_i$ corresponds to $Y_i$ and $\partial^{|j_k|} z_{i_k}/\partial x^{j_k} $ corresponds to $Y_{q+k}$. Then applying Corollary \ref{cor}  to $f$ we get a function $\tau:{\bf N}^{q+s}\to {\bf N}$ which works also in our case $F$.
\hfill\ \end{proof}

{\bf Acknowledgements:} We thank the referee of their relevant and helpful comments.\\
This work has been partially elaborated in the frame of the International
Research Network ECO-Math.

\end{document}